\documentclass{amsart}

\usepackage{amsmath,amssymb,amsfonts,amstext,amsthm,latexsym,ifthen}
\usepackage{graphicx}
\usepackage{tikz}
\usepackage[all,cmtip]{xy} 
\usepackage{multicol}
\input xy
\xyoption{all}
\usepackage{algpseudocode}

\newcommand{\nwc}{\newcommand}

\nwc{\aaa}{\mathcal{F}}
\nwc{\aap}{\mathcal{F}_{P}}

\nwc{\cb}{\overline{C}}
\nwc{\ccc}{\mathfrak{c}}
\nwc{\ch}{\widehat{C}}
\nwc{\cin}{\textbf{(v)}}
\nwc{\cpll}{\mathfrak{c}_{P'}}
\nwc{\ct}{\tilde{C}}

\nwc{\dd}{\text{L}}
\nwc{\ddd}{\mathfrak{d}}
\nwc{\ddl}{\mathcal{L}'}
\nwc{\dlp}{\delta_{P}}
\nwc{\doi}{\textbf{(ii)}}

\nwc{\enq}{$ $}
\nwc{\edim}{{\rm edim}}

\newcommand{\F}{\mathbb{F}} 
\newcommand{\Q}{\mathbb{Q}}

\nwc{\gtl}{\tilde{g}}
\nwc{\hua}{h^{1}(C,\aaa )}

\nwc{\llb}{\mathcal{L}}
\nwc{\lpl}{\ell_{{\rm P}'}}

\nwc{\mm}{\mathfrak{m}}
\nwc{\mmp}{\mathfrak{m}_{P}}
\nwc{\mmq}{\mathfrak{m}_{Q}}
\nwc{\mpd}{\mathfrak{m}_{P}^{2}}

\nwc{\nn}{\mathbb{N}}

\nwc{\ob}{\overline{\mathcal{O}}}
\nwc{\obp}{\overline{\mathcal{O}}_P}
\nwc{\och}{\mathcal{O}_{\hat{C}}}
\nwc{\oh}{\hat{\mathcal{O}}}
\nwc{\ohp}{\hat{\mathcal{O}}_{P}}
\nwc{\ol}{\mathcal{O}'}
\nwc{\oma}{\Omega (\mathfrak{a})}
\nwc{\omo}{\Omega (\mathcal{O})}
\nwc{\oo}{\mathcal{O}}
\nwc{\op}{\mathcal{O}_P}
\nwc{\opc}{\mathcal{O}_{P,C}}
\nwc{\oph}{\hat{\mathcal{O}}_{P}}
\nwc{\opl}{\mathcal{O}_{P}'}
\nwc{\oplc}{\mathcal{O}_{P,C}'}
\nwc{\opll}{\mathcal{O}_{P'}}
\nwc{\opt}{\tilde{\mathcal{O}}_{P}}
\nwc{\optt}{{\mathcal{O}}_{\tilde{P}}}
\nwc{\oq}{\mathcal{O}_{Q}}
\nwc{\oqt}{\tilde{\mathcal{O}}_{Q}}
\nwc{\ot}{\tilde{\mathcal{O}}}
\nwc{\overop}{\bar{\oo}_{P}}

\nwc{\pb}{\overline{P}}
\nwc{\pgmd}{\mathbb{P}^{g+2}}
\nwc{\pgmu}{\mathbb{P}^{g+1}}
\nwc{\ph}{\hat{P}}
\nwc{\pp}{\mathbb{P}}
\nwc{\prv}{\noindent\textbf{Proof}:}
\nwc{\pt}{\tilde{P}}
\nwc{\ptl}{\tilde{P}}
\nwc{\pum}{\mathbb{P}^{1}}

\nwc{\qh}{\hat{Q}}
\nwc{\qtl}{\tilde{Q}}
\nwc{\qua}{\textbf{(iv)}}

\nwc{\rh}{\hat{R}}

\nwc{\sei}{\textbf{(vi)}}
\nwc{\sep}{\beq\ast\ \ast\ \ast\enq}
\nwc{\spc}{{\rm Spec}}
\nwc{\ssp}{S_{P}}

\nwc{\tre}{\textbf{(iii)}}

\nwc{\um}{\textbf{(i)}}

\nwc{\vpb}{v_{\overline{P}}}

\nwc{\wh}{\hat{\omega}}
\nwc{\whp}{\hat{\omega}_{P}}
\nwc{\woch}{\omega\cdot\mathcal{O}_{\hat{C}}}
\nwc{\woh}{\omega\cdot\hat{\mathcal{O}}}
\nwc{\ww}{\omega}
\nwc{\wwh}{\widehat{\omega}}
\nwc{\wwhp}{\widehat{\omega}_P}
\nwc{\wwp}{\omega _{P}}

\nwc{\zz}{\mathbb{Z}}

\newcommand{\C}{\mathbb{C}}
\newcommand{\Z}{\mathbb{Z}}

\newtheorem{thm}{Theorem}
\newtheorem{theorem}{Theorem}[section]

\newtheorem{prop}[theorem]{Proposition}
\newtheorem{lemma}[theorem]{Lemma}
\newtheorem{corollary}[theorem]{Corollary}

\theoremstyle{definition}

\newtheorem{definition}[theorem]{{\bf Definition}}

\newcommand{\alt}[1]{{\sf A}_{#1}}

\newcommand{\primei}[2]{#1^{(#2)}}
\newcommand{\primeins}[2]{\overline{#1}^{(#2)}}

\renewcommand{\leq}{\leqslant}
\renewcommand{\geq}{\geqslant}

\newcommand{\rbar}[1]{\overline R(#1)}
\newcommand{\rtilde}[1]{\widetilde R(#1)}
\newcommand{\spec}[1]{{\rm Spec}(#1)}

\newcommand{\cl}[1]{{\sf Cl}(#1)}

\begin{document}

\title[Singularities of representation rings]{Computing singularities
  of the spectra of representation rings of finite groups}

\author[A. G. Bueno]{Andr\'e Gimenez Bueno}
\address[Bueno, Martins, Schneider]{ICEx - UFMG \\
Departamento de Matem\'atica \\ 
Av. Ant\^onio Carlos 6627 \\
30123-970 Belo Horizonte MG, Brazil}

\author[R. V. Martins]{Renato Vidal Martins}

\author[E. Oliveira]{Edney Oliveira}
\address[Oliveira]{ICEB - UFOP \\
Departamento de Matem\'atica \\
Campus Morro do Cruzeiro\\
35400-000 Ouro Preto, MG, Brazil} 

\author[C. Schneider]{Csaba Schneider}

%


\begin{abstract}
Let $G$ be a finite group of order $n$, and $\xi$ an $n$-th primitive root of unity. 
Consider the affine scheme $C:=\spc(\Z[\xi]\otimes_\Z R(G))$ where $R(G)$ is the representation ring of $G$. We study the fibers of the formal tangent sheaf of $C$ by computing their dimension and also finding (and measuring) the singularities of $C$. We present explicit computations for noncommutative groups of small order, and develop practical methods to compute these invariants for an arbitrary finite group. 
\end{abstract}


\maketitle

\section*{Introduction}\label{sec:intro}

Let $G$ be a group, $A$ a commutative ring, and $A[G]$ the associated group ring. It is well known that $A[G]$ is a Hopf algebra which represents a diagonalizable affine group scheme over $A$ (see, for instance, \cite[Ch. 3]{CS}, \cite{D} and \cite{W}). If $G$ is abelian, then $A[G]$ is commutative and it is natural to study the affine scheme $X:=\spc(A[G])$. This study is rather simple when
$A$ is a field, and becomes even simpler  when $G$ is also
finite in which case the irreducible components of $X$ are points.

So the next step is considering group rings $\Z[G]$ for abelian groups $G$. Then the complexity increases considerably even if the concern is only set-theoretical. Indeed, the topological structure of $\spc(\Z[G])$ is non-trivial in the sense that its irreducible components may intersect in several points. If, in addition, one is interested in taking the problem towards a refined scheme-theoretical framework, then many additional problems may arise. One of them is identifying and describing singularities by means of their natural invariants, such as,
tangent spaces, formal embedding dimensions, and ramification indices.  This was done by the first named author in collaboration with M.~Dokuchaev in \cite{BD,BD1} (based on \cite{B}), where a thorough description of the scheme $\spc(\Z[G])$ is presented and its singular points are described.

Removing the hypothesis of $G$ being abelian is the point of departure of the present work. Our aim is to obtain similar results to the ones of \cite{BD,BD1} in the case where $G$ is noncommutative. However, we do have to preserve the commutativity of the resulting ring, so we naturally replace $\Z[G]$ with
the \emph{representation ring} (sometimes called the {\em Green ring}) $R(G)$ of $G$,
  since $R(G)\cong \Z[G]$ when $G$ is abelian. Topological descriptions of
  $\spc(R(G))$ were obtained, among others, by M.~Atiyah \cite{A}, G.~Segal \cite{Sg} and  J.~P.~Serre \cite{Sr}. It is clear from these works that, instead of
  dealing with $\spc(R(G))$, 
  it is more natural to consider 
  $C:=\spc(\Z[\xi]\otimes_\Z R(G))$ where $\xi$ is a primitive $n$-th root of  unity and $n$ is the order of $G$.
  The study of $C$ as a scheme is precisely the objective of the present paper. 

  In Sections~\ref{sec:cyclo} and~\ref{sec:jac} we present some results concerning the geometry of cyclotomic rings and the relation between the Zariski tangent space and the Jacobian matrix over domains.  In Section~\ref{sec:Green} we introduce the ring $\rbar G:=R(G)\otimes_\Z\Z[\xi]$, where $R(G)$ is the
  representation ring of $G$,
  and summarize some topological properties of $C=\spec{\rbar G}$. 
  Using the topological descriptions of the aforementioned articles, we
  analyze the scheme-theoretical decomposition of $C$
  into irreducible components. 

  Afterwards, we study the singular points of $C$. In doing so, it is important to compare the dimension of the Zariski tangent space (defined by fibers of the tangent sheaf) and the (formal) embedding dimension at a point; these dimensions coincide if the base scheme is the spectrum of a field isomorphic to the residue field, which is a very common situation. But this may not be the case in general, particularly when we  deal with schemes over $\Z$. 

  Given a point $P$ on a scheme $X$, we denote by $T_P(X/Y)$ the \emph{Zariski tangent space} at $P$ if $X$ is a scheme over $Y$. One may abuse the notation and write $T_P(X/A)$ in case $Y=\spc(A)$. We also recall the definition of the \emph{embedding dimension} $\edim_P(X):=\dim_{k_P}(\mmp/\mmp^2)$, where $\mmp$ is the maximal ideal and $k_P$ the residue field (see, respectively, Sections ~\ref{sec:jac} and ~\ref{sec:Green} for more details).

    The following main theorem describes the precise relationship between
  the embedding dimension of $C$, and the dimensions of the Zariski
  tangent spaces of $C$ over $\Z$ and over $\Z[\xi]$ at a point $P$.

  \begin{thm}\label{th:edim}
    Let $G$ be a finite group of order $n$ and $\xi$ a primitive
    $n$-th root of unity in $\C$. Set $C:=\spec{R(G)\otimes\Z[\xi]}$
    where $R(G)$ is the representation ring of $G$. For any point $P\in C$, let $p$ be the prime number such that $p\in P\cap\Z$.
    Then the following hold.
    \begin{enumerate}
    \item[(i)] If $p\mid n$ and either $p\neq 2$ or $4\mid n$, then
      $$
      \dim_{k_P}T_P(C/\Z[\xi])+1=\dim_{k_P}T_P(C/\Z)= \edim_P(C).
      $$
    \item[(ii)] If $p\nmid n$ or $p=2$ and $4\nmid n$, then
      $$
      \dim_{k_P}T_P(C/\Z[\xi])=\dim_{k_P}T_P(C/\Z)= \edim_P(C)-1.
      $$
    \end{enumerate}
    Furthermore, (i) is valid if and only if $p\in P^2$.
  \end{thm}

  The actual values of the dimensions in Theorem~\ref{th:edim} can be
  calculated using a suitable Jacobian matrix; see Section~\ref{sec:comp}.
  
As was done  in \cite{BD,BD1,B}, we also address here the case where $G$ is commutative in Corollary \ref{corabl}, where we derive a formula for the dimension of the tangent space by means of the cyclic group decomposition available for abelian groups and the general results we obtained for arbitrary groups. 

Based on Theorem~\ref{th:edim}, in Section~\ref{sec:comp}, we present a practical method for the computation of the dimensions of the Zariski tangent spaces and the embedding
dimension for finite groups. The method can be presented in the form of an algorithm which was implemented in the computational algebra system Magma~\cite{Magma}. The implementation enabled us to list all singular points and their invariants of the spectra of representation rings
defined by noncommutative  groups of larger order.

\subsection*{Acknowledgement}
The results presented in this paper were obtained during the PhD course
of the third author under the supervision by the first and the second authors
partially funded by a CNPq PhD scholarship
(process number: 141115/2011-4)
at the {\em Universidade Federal de Minas Gerais} in Belo Horizonte (Brazil);
Oliveira is also grateful to the {\em Universidade Federal de Ouro Preto}
for their support.
In addition, Martins and Schneider were supported by the CNPq projects
{\em Produtividade em Pesquisa}
(project no.: 305240/2018-8 and 308773/2016-0) 
and {\em Universal} (project no.: 421624/2018-3). We thank Ethen Cotterill and
Eduardo Tengan for their suggestions.

\section{The geometry of cyclotomic extensions}\label{sec:cyclo}

The main objective of this paper is to describe the singularities of
the spectra of rings
that are related to cyclotomic extensions $\Z[\xi]$ of $\Z$. In this section, we
summarize some  results concerning $\Z[\xi]$ that will be necessary in the
subsequent sections. As a general reference on this topic, the reader is recommended
to consult~\cite[Chapter~1]{Neukirch}.

Let $n\geq 2$ and suppose that $\Phi_n(x)\in\Z[x]$ is the $n$-th cyclotomic polynomial. It is
well known that $\Phi_n(x)$ is an irreducible polynomial in $\Z[x]$. Set
$\Z[\xi]=\Z[x]/(\Phi_n(x))$; that is, $\Z[\xi]$ is the extension of $\Z$ by
a primitive $n$-th root of unity $\xi$. It is well-known that $\Z[\xi]$
coincides with the ring of integers of the cyclotomic field
$\Q(\xi)$; see~\cite[(10.2) Proposition]{Neukirch}.

By \cite[(10.3) Proposition]{Neukirch}, if $p$ is a prime and $p^\alpha$ is
the largest $p$-power that divides $n$, then $\Phi_n(x)$ can be factorized
in $\F_p[x]$ as
$$
\Phi_n(x)=f_1(x)^{\varphi(p^\alpha)}\cdots f_r(x)^{\varphi(p^\alpha)}
$$
where $f_1(x),\ldots,f_r(x)\in\F_p[x]$ are irreducible polynomials
of the same degree.
Furthermore, the  ideal $(p)$ of $\Z[\xi]$ admits the factorization
\begin{equation}\label{eq:factp}
(p)=(f_1(\xi))^{\varphi(p^\alpha)}\cdots (f_r(\xi))^{\varphi(p^\alpha)}.
\end{equation}
Note that the common exponent $\varphi(p^\alpha)$ appearing in these
expressions is greater than one if and only
if
\begin{equation}\label{eq:divisibility}
  p\mid n\quad \mbox{and either}\quad p\neq 2\mbox{ or }4\mid n,
  \end{equation}
and this is the reason
why this divisibility condition appears in Theorem~\ref{th:edim}.

The scheme $\spec{\Z[\xi]}$ can be described as
\begin{align*}
  \spec{\Z[\xi]}&=\big\{(0)\big\}\cup \big\{(p,f(\xi))\ \big{|}\ \mbox{$p\in\Z$ is a prime and}\\
& \ \ \ \ \ \ \ \ \ \  \ \ \ \ \ \ \ \ \ \ \ \ \ \ \ \ \ \ \ \mbox{ $f(x)$ is an irreducible factor of $\Phi_n(x)$ 
  in $\F_p[x]$}\,\big\}.
\end{align*}
In particular, each non-zero prime of $\Z[\xi]$ is maximal and $\dim\Z[\xi]=1$.
  The following result  describes the local structure of $\Z[\xi]$.

\begin{lemma}\label{lem:cyclomult}
  Let $Q=(p,f(\xi))$ be a closed point in $\spec{\Z[\xi]}$, set $k:=\Z[\xi]/Q$
  and let $\zeta=\xi+Q$ so that $\zeta$ is a root of $\Phi_n(x)$ in $k$.
  Then
  $$
  \dim_k Q/Q^2=1,
  $$
  and, in particular,  $\spec{\Z[\xi]}$ is smooth.
  Furthermore, the following hold.
  \begin{enumerate}
  \item [(i)] If condition~\eqref{eq:divisibility} is valid, then
    $p\in Q^2$ and so $Q/Q^2=(f(\xi)+Q^2)$, Moreover, $\Phi_n'(\zeta)=0$.
    \item [(ii)] If condition~\eqref{eq:divisibility} is not valid, 
    then $Q/Q^2=(p+Q^2)$, and $\Phi_n'(\zeta)\neq 0$.
  \end{enumerate}
\end{lemma}
\begin{proof}
 Since $\Z[\xi]$ coincides with the 
 ring of integers of the cyclotomic field $\Q(\xi)$, the ring
 $\Z[\xi]$ is a one-dimensional noetherian normal domain. This
implies that the unique maximal ideal $\mm$ of the localization $\Z[\xi]_Q$ is 
principal and $\dim_k \mm/\mm^2=1$ where $k=\Z[\xi]/Q\cong (\Z[\xi]_Q)/\mm$ is the residue field (see~\cite[Theorem~11.2]{M}).  Also note that, since $Q$ is a maximal ideal of $\Z[\xi]$, we have that 
$\mm/\mm^2\cong Q/Q^2$.

Let us prove statements~(i) and~(ii). Assume without loss of generality
that $f(x)=f_1(x)$ in expression~\eqref{eq:factp}.
If condition~\eqref{eq:divisibility} holds,
then $\varphi(p^\alpha)>1$, and so, by~\eqref{eq:factp}, $p$ is contained in the ideal $(f(\xi)^2)$. Thus, in this case, $p$ is contained in $Q^2$ and
$Q/Q^2$ must be generated by
$f(\xi)+Q^2$. Also, $\zeta\in k$ is a multiple root of $\Phi_n(x)$, and so
$\Phi_n'(\zeta)=0$.
If, on the other hand, condition~\eqref{eq:divisibility} does not
hold, then $\varphi(p^\alpha)=1$ and the elements $f_i(\xi)$ with $i\geq 2$
are invertible in  $\Z[\xi]_Q$. Hence in this case,
$f(\xi)\in (p)\subseteq \Z[\xi]_Q$ and in particular, $Q/Q^2$ is generated by $p+Q^2$. In this case $\Phi_n(x)$ has no multiple roots in $k$, and hence
$\Phi_n'(\zeta)\neq 0$.
\end{proof}

Lemma~\ref{lem:cyclomult} will be very useful when we start comparing the
embedding dimension and tangency at a given point of the spectrum
of the representation ring of a finite group. So, for the sake of later use, we opted to state it in this section.

\section{Tangent Spaces and Jacobians}\label{sec:jac}

In order to compute the main invariants of a singular point of the spectra  associated with representation rings of finite groups, we need a theorem that links the dimension of the Zariski tangent space to the Jacobian  matrix. In this section, we state and prove Theorem~\ref{th:jac} due to the lack of references
in the literature for such a result that is valid over domains.

We start by recalling the scheme-theoretic definition of the tangent space at a point on a given scheme (see \cite[Def. VII.5.4]{AK}).

\begin{definition}
\label{defted}
Let $X$ be a scheme over $Y$ and $P\in X$ be a point. Let also $\oo_P$ be the local ring of $P$ on $X$ and $\mmp$ its maximal ideal. Set $k_P:=\oo_P/\mmp$ to be the residue field. The \emph{Zariski tangent space} of $X/Y$ at $P$ is defined by
$$
T_PX=T_P(X/Y):={\rm Hom}_{k_P}(\Omega_{X/Y}\otimes_{\oo_X}k_P,k_P)
$$
where $\Omega_{X/Y}$ is the sheaf of differentials. 
\end{definition}

In Definition~\ref{defted},  $\Omega_{X/Y}\otimes_{\oo_X}k_P$ is a skyscraper sheaf supported at $P$ and is identified here with its stalk at $P$, which is a $k_P$ vector space, and so the definition makes sense.

The main result of this section claims that under rather general conditions,
the dimension of $T_P X$ can be computed by computing the rank of a suitable
Jacobian matrix.

\begin{theorem}\label{th:jac}
Let $A[\mathbf x]=A[x_1,\ldots,x_s]$ be a polynomial ring over a domain $A$, let
$I=(f_1,\ldots,f_r)$ be an ideal of $A[\mathbf x]$ and set
$R=A[\mathbf x]/I$. For $P\in\spec R$, let $k_P$ be the residue field of $P$.
For $f\in A[\mathbf x]$, let $\widehat f$ be the 
corresponding element in $k_P[\mathbf x]$ and let $\overline{f}\in k_P$ be the residue class of $f$. Consider the ideal $
\widehat I =(\widehat f_1,\ldots,\widehat f_r) \subseteq k_P[\mathbf x]
$, and let $\widehat{P}$ be the closed point in $\spec {k_P[\mathbf x]/\widehat I}$ corresponding to $\alpha :=(\overline{x_1},\ldots,\overline{x_s})\in k_P^s$. Set
    $$
  J:=\begin{pmatrix}
  \displaystyle\frac{\partial \widehat f_1}{\partial x_1}(\alpha) & \cdots &
   \displaystyle \frac{\partial \widehat f_1}{\partial x_s}(\alpha)\\
    \vdots  & \ddots & \vdots\\
    \displaystyle\frac{\partial \widehat f_r}{\partial x_1}(\alpha) & \cdots &
    \displaystyle\frac{\partial \widehat f_r}{\partial x_s}(\alpha)\end{pmatrix}
  $$
  to be the Jacobian matrix.
  Denoting by $\mm_{\widehat P}$ the unique maximal ideal of the localization $(k_P[\mathbf x]/\widehat I)_{\widehat P}$, we have that
  $$
  \dim_{k_P} T_P(\spec R/A)=\dim_{k_P} \mm_{\widehat P}/\mm_{\widehat P}^2=s-{\rm rank}\,J.
  $$
\end{theorem}
\begin{proof}
As is well-known, $\Omega_{A[\mathbf x]/A}$ is a free $A[\mathbf x]$-module
generated by $dx_1,\ldots,dx_s$; that is,
$$
\Omega_{A[\mathbf x]/A}=A[\mathbf x]dx_1\oplus\cdots
\oplus A[\mathbf x]dx_s.
$$
Furthermore, there is an exact sequence 
\begin{equation}\label{eq:exseq}
I/I^2\longrightarrow \Omega_{A[\mathbf x]/A}\otimes_{A[\mathbf x]}R\longrightarrow
\Omega_{R/A}\longrightarrow 0
\end{equation}
of $R$-modules where the first map
$I/I^2\rightarrow \Omega_{A[\mathbf x]/A}\otimes_{A[\mathbf x]}R$ is given by $f+I^2\mapsto df\otimes 1$.
Suppose that $P\in\spec R$  and let $R_P$ be the
localization of $R$ at $P$. Then $R_P$ is a local ring with unique maximal
ideal $\mm_P$ and with residue field $k_P=R_P/\mm_P$. 
Localizing the right exact sequence 
\eqref{eq:exseq} at $P$ and then applying $-\otimes_{R_P}k_P$, we obtain
$$
  (I/I^2)_P\otimes_{R_P}k_P
  \longrightarrow \Omega_{A[\mathbf x]/A}\otimes_{A[\mathbf x]}R_P\otimes_{R_P}k_P\
  \longrightarrow
\Omega_{R_P/A}\otimes_{R_P} k_P\longrightarrow 0
$$
which reduces to
$$
  (I/I^2)_P\otimes_{R_P}k_P
  \longrightarrow \Omega_{A[\mathbf x]/A}\otimes_{A[\mathbf x]}k_P\
  \longrightarrow
\Omega_{R_P/A}\otimes_{R_P} k_P\longrightarrow 0.
$$

Noting that $(I/I^2)_P\otimes_{R_P}k_P$ is generated by
$(f_1+I^2)\otimes 1,\ldots,(f_r+I^2)\otimes 1$, we can write
$\Omega_{R_P/A}\otimes_{R_P} k_P$ as 

$$
 \Omega_{R_P/A}\otimes_{R_P} k_P\cong\frac{(A[\mathbf x]dx_1\oplus\cdots
   \oplus A[\mathbf x]dx_s)\otimes_{A[\mathbf x]} k_P}{
   (df_1\otimes  1,\ldots,df_r\otimes  1)_{R_P}}
.
$$
 If $f\in A[\mathbf x]$, then $df=\sum_j(\partial f/\partial x_j) dx_j$,
 which implies that the $R_P$-module in the denominator of the last displayed
equation is generated by
$$
\left(\sum_{j=1}^s(\partial f_1/\partial x_j) dx_j\right)\otimes  1,
\ldots,\left(\sum_{j=1}^s
(\partial f_r/\partial x_j)dx_j\right)\otimes 1.
$$
Moreover, for $i\in\{1,\ldots,s\}$,
\begin{align*}
\sum_{j=1}^s\left(\frac{\partial f_i}{\partial x_j}dx_j\otimes 1
\right)&=
\sum_{j=1}^s \left(dx_j\otimes \overline{\frac{\partial f_i}{\partial x_j}}\right)
\\&=
\sum_{j=1}^s\left(dx_j\otimes \frac{\partial\widehat f_i}{\partial x_j}(\alpha)\right)=
\sum_{j=1}^s\frac{\partial\widehat f_i}{\partial x_j}(\alpha)\left(dx_j\otimes 1\right).
\end{align*}
This gives that
$$
\Omega_{R_P/A}\otimes_{R_P} k_P\cong
W/S
$$
where
$$
W=(dx_1\otimes 1,\ldots, dx_s\otimes 1)_{k_P}
$$
and
$$
S=
     {\left(\sum_{j=1}^s \frac{\partial\widehat f_1}{\partial x_j}(\alpha)
       (dx_j\otimes 1),
       \ldots,\sum_{j=1}^s \frac{\partial\widehat f_r}{\partial x_j}(\alpha)
       (dx_j\otimes 1)
     \right)_{k_P}}.
     $$
Since $A$ is a domain, $dx_1\otimes 1,\ldots, dx_s\otimes 1$
     are linearly independent over $k_P$, and so $W\cong k_P^s$. On the other hand, any $\omega\in S$ can be written as
$$
\omega= (dx_1\otimes 1\ \ldots\ dx_s\otimes 1)
 J^T\begin{pmatrix}
  a_1 \\
  \vdots\\
 a_r
  \end{pmatrix}
$$
where $J^T$ is the transpose of the Jacobian matrix and $(a_1,\ldots,a_r)\in k_P^r$. Thus $\dim_{k_P} S=\mbox{rank}\,J$ and 
    $$
  \dim_{k_P} W/S=s-\mbox{rank}\,J.
  $$
  This proves the equation $\dim_{k_P} T_P(\spec R/A)=s-\mbox{rank}\,J$.

Now set $k_P[\mathbf x]=k_P[x_1,\ldots,x_s]$ and let us define the $k_P$-homomorphism
\begin{gather*}
\begin{matrix}
\psi:	&k_p[\mathbf x]& \longrightarrow & W  \\
		& f  & \longmapsto     & \displaystyle  \sum_{j=1}^s\left(\frac{\partial f}{\partial x_j}(\alpha)
(dx_j\otimes 1)\right).
		\end{matrix}
\end{gather*}
Set also $\alpha_i:=\overline{x_i}$ for all $i\in\{1,\ldots,s\}$ so that $\alpha=(\alpha_1,\ldots,\alpha_s)$ and consider the ideal
$$
M_\alpha:=(x_1-\alpha_1,\ldots,x_s-\alpha_s)\subseteq k_P[\mathbf x].
$$ 
 Since
     $\psi(x_i-\alpha_i)=dx_i\otimes 1$, we have that $\psi|_{M_\alpha}$ is surjective.  Let us show that $M_\alpha\cap\ker\psi=M_\alpha^2$.
     If $f,g\in M_\alpha$, then clearly $\partial (fg)/\partial x_i(\alpha)=0$ due to Leibniz's rule and the fact that $f(\alpha)=g(\alpha)=0$. 
     Hence $\psi(fg)=0$, and so $M_\alpha^2\subseteq\ker\psi$.
    Now suppose $f\in M_\alpha$ and $\psi(f)=0$. As
     $f=g_1(x_1-\alpha_1)+\cdots+g_s(x_s-\alpha_s)$, we have that 
     $$
  0=\frac{\partial f}{\partial x_i}(\alpha)= \sum_{j=1}^s\left(\frac{\partial g_j}{\partial x_i}(x_j-\alpha_j)-
  g_j\frac{\partial(x_j-\alpha_j)}{\partial x_i}\right)(\alpha)=
  g_i(\alpha).
  $$
  Thus $g_i(x)\in M_\alpha$ for all $i$, and hence $f\in M_\alpha^2$.  This shows that
  $\psi$ induces an isomorphism $M_\alpha/M_\alpha^2\cong W$ and hence
  $\dim_{k_P} M_\alpha/M_\alpha^2=\dim_{k_P} W=s$.

  Setting $\widehat I =(\widehat f_1,\ldots,\widehat f_r)$, we
  note that $\widehat I\subseteq M_\alpha$. Furthermore, if $f\in\widehat I$ then $f=g_1\widehat f_1+\cdots+g_r\widehat f_r$ and hence
  \begin{align*}
  \psi(f)&=\sum_{i=1}^r\frac{\partial \widehat f_i}{\partial x_1}(\alpha)g_i(\alpha)
  (dx_1\otimes
  1)+\cdots+
  \sum_{i=1}^r\frac{\partial \widehat f_i}{\partial x_s}(\alpha)g_i(\alpha)
  (dx_s\otimes
  1)\\&=
  (dx_1\otimes 1\ \ldots\ dx_s\otimes 1)
 J^T\begin{pmatrix}
  g_1(\alpha) \\
  \vdots\\
  g_r(\alpha)
  \end{pmatrix}
    \end{align*}
so 
  $$
  \dim_{k_P}\psi(\widehat I)=\dim_{k_P} (\widehat I+M_{\alpha}^2)/M_{\alpha}^2=\mbox{rank}\,J.
  $$
Therefore
 \begin{align*}
  s&=\dim_{k_P} M_{\alpha}/M_{\alpha}^2\\
    &=\dim_{k_P} M_{\alpha}/(\widehat I+M_{\alpha}^2)+\dim_{k_P} (\widehat I+M_{\alpha}^2)/M_{\alpha}^2\\
    &=\dim_{k_P} ((M_{\alpha}/\widehat{I})/((\widehat I+M_{\alpha}^2)/\widehat I)+\mbox{rank}\,J\\
    &=\dim_{k_P} ( \mm_{\widehat P}/\mm_{\widehat P}^2)+\mbox{rank}\,J
 \end{align*}
and the proof is complete.
\end{proof}



As was already pointed out, Theorem~\ref{th:jac} combined with
our results  in Sections~\ref{sec:Green} and~\ref{sec:comp}
enables us to explicitly calculate the invariants
of singularities for representation rings of finite groups.

\section{The geometry of the spectra of representation rings}\label{sec:Green}

The main subject of this article is the affine scheme associated with
the representation ring of a finite group. In this section, whose first part is largely
based on Serre's text~\cite[Section~11.4]{Sr}, we give an overview
of its structure. This envolves a brief study of the natural interactions of
cyclotomic extensions with representation rings.


Let $G$ be a finite group of order $n$ and let $\xi$ be a primitive
$n$-th roof of unity in $\C$. Letting $\cl G$ denote the set of
conjugacy classes of $G$, a
complex character $\chi$ of $G$ can be viewed as a function
$$
\chi:\cl G\rightarrow \Z[\xi].
$$
Suppose that $\chi_1,\chi_2,\ldots,
\chi_s$ are the irreducible complex characters of $G$. The \emph{representation ring} $R(G)$ is the set of
$\Z$-linear combinations of $\chi_1,\ldots,\chi_s$ viewed as a ring
under the product operation between characters. Then $R(G)$ is a commutative
ring whose identity element is the trivial character, which we shall assume throughout this work to be $\chi_1$. The elements of $R(G)$ are sometimes called {\em virtual characters} of $G$. The additive group of  $R(G)$ coincides with the Grothendieck
group of the category of finitely generated $\C[G]$-modules.

Our main concern here is the study of the ring 
$$\rbar G:=R(G)\otimes_\Z
\Z[\xi]
$$
or, more precisely, the affine scheme
$$
C:=\spec {\rbar G}.
$$
Note that $
\rbar G$ can be viewed as the set of $\Z[\xi]$-linear combinations
of $\chi_1,\ldots,\chi_s$ and it is also a commutative ring with the identity
element $\chi_1$. 
The ring $\rbar G$ can also  be seen as a subring of $\Z[\xi]^{\cl G}$ of
functions from $\cl G$ to $\Z[\xi]$, which is naturally
isomorphic to the direct sum $\Z[\xi]\oplus\cdots\oplus \Z[\xi]$
of $|\cl G|$ copies of $\Z[\xi]$.

The embeddings
\begin{equation}\label{eq:ringembed}
  \Z\rightarrow \Z[\xi]\rightarrow\rbar G\rightarrow \Z[\xi]^{\cl G}
\end{equation}
  defined by
  $\xi \mapsto \chi_1\otimes\xi$ and  $\chi_i\otimes\xi\mapsto
  \xi\chi_i$, respectively, 
induce maps of schemes
\begin{equation}\label{eq:specproj}
  \spec{\Z[\xi]^{\cl G}}\rightarrow C \rightarrow\spec{\Z[\xi]}\rightarrow
  \spec{\Z}.
\end{equation}
Since $\Z[\xi]^{\cl G}$ is a finitely generated $\Z$-module,
we have that
the extension $\Z\leq \Z[\xi]^{\cl G}$ is integral and so the maps
in~\eqref{eq:specproj} are surjective.

For $Q\in\spec{\Z[\xi]}$ and $c\in\cl G$, set
$$
\primei Qc :=\{\chi:\cl G\rightarrow\Z[\xi]\mid \chi(c)\in Q\}.
$$
Then 
$$
\spec{\Z[\xi]^{\cl G}}=\{\primei Qc\mid Q\in\spec{\Z[\xi]}\mbox{ and } c\in \cl G\}.
$$
For a prime $\primei Qc\in\spec{\Z[\xi]^{\cl G}}$, let 
$$
\primeins Qc := \primei Qc\cap \rbar G.
$$
Considering the maps in~\eqref{eq:ringembed} as embeddings,
we obtain that
$$
C=\{\primeins Qc\mid Q\in\spec{\Z[\xi]}\mbox{ and }c\in\cl G\}.
$$
In particular, as the minimal primes of $\Z[\xi]^{\cl G}$
are of the form
$$
\primei{(0)}c=\{\chi\in \Z[\xi]^{\cl G}\mid \chi(c)=0\}\quad\mbox{where}\quad c\in\cl G,
$$
the minimal primes of $\rbar G$ are the ideals
$\primeins {(0)}c=\primei{(0)}c\cap \rbar G$ with $c\in\cl G$.
This argument also implies that $\dim \rbar G=1$.

Let $P=\primeins Qc$ be a point of $C$ with $Q\neq(0)$, let
$\mm_P$ be the unique maximal ideal of the localization
$\rbar G_P$, and set $k_P$ to be the residue field of $P$.
Then $\Z[\xi]$ can naturally be identified with
  $\{\alpha\chi_1\mid \alpha\in\Z[\xi]\}$ inside $\rbar G$, and so
  $\alpha\chi_1\in P$ if and only if $\alpha\chi_1(c)=\alpha\in Q$. Hence
  $P\cap \Z[\xi]=Q$. Furthermore, since $P$ is a maximal ideal,
\begin{equation}\label{eq:kp}
  k_P\cong \Z[\xi]/Q.
  \end{equation}
In particular, $k_P$ is a finite field.

Recall that if $g\in G$ and $p$ is a prime, then
$g$ can be written uniquely as $g=g_u^{(p)}g_r^{(p)}$ where
$g_u^{(p)},g_r^{(p)}\in\left<g\right>$, the order of $g_u^{(p)}$
is a $p$-power, while $p$ does not divide the order of $g_r^{(p)}$.
An element $g$ is said to be {\em $p$-regular} if $p\nmid |g|$.
If $g$ is $p$-regular, then $g_u^{(p)}=1$ and $g_r^{(p)}=g$. If $c\in \cl G$
is a conjugacy class, then $c_r^{(p)}=\{g_r^{(p)}\mid g\in c\}$ is a conjugacy
class. A class $c$ is said to be $p$-regular if it consists of
$p$-regular elements of $g$. If $c$ is $p$-regular, then clearly
$c_r^{(p)}=c$.

For the proof of the following result,
see \cite[Lem. 15.2]{isaacs} and \cite[Prp.~30]{Sr}.

\begin{lemma}\label{lem:charval}
  Let $G$ be a finite group and let $Q$ be a maximal ideal of $\Z[\xi]$ such that
  ${\rm char}(\Z[\xi]/Q)=p$. Then the following are equivalent for
  $c,d\in\cl G$:
  \begin{itemize}
    \item[(i)] $\chi(c)\equiv \chi(d)\ {\rm mod}\ Q$ for all characters $\chi$ of $G$;
    \item[(ii)] $c^{(p)}_r=d^{(p)}_r$.
  \end{itemize}
  \end{lemma}

Lemma~\ref{lem:charval}  is the essence of a scheme-theoretic description of
$C$, given as follows.

\begin{lemma} 
\label{prpggr}
Suppose that $c,d\in\cl G$ are distinct classes of $G$. Then the following hold.
  \begin{itemize}
  \item[(i)] $\primeins {(0)}{c}\neq \primeins {(0)}{d}$. In particular, there is a one-to-one correspondence between the set of irreducible components of $C$ and the set of conjugacy classes
    of $G$.
  \item[(ii)] If $Q,Q'\in\spec{\Z[\xi]}\setminus\{(0)\}$, then
    $\primeins Qc=\primeins{Q'}d$ if and only if $Q=Q'$ and 
    $c^{(p)}_r=d^{(p)}_r$. 
  \end{itemize}
\end{lemma}
\begin{proof}
Both statements (i) and (ii) follow from Lemma~\ref{lem:charval}. 
 \end{proof}

Lemma~\ref{prpggr} leads to the following corollary.

\begin{corollary}\label{cor:comps}
  Using the notation of Lemma~\ref{prpggr},
  all irreducible components of $C$ are homeomorphic to
  $\spec{\Z[\xi]}$. Furthermore, for $c\in\cl G$ and $Q\in\spec{\Z[\xi]}$,
  the prime $P=\primeins Qc$ is contained in
  the intersection of two irreducible components of $C$ if and only if
  $Q\neq (0)$ and $c_r^{(p)}=d_r^{(p)}$ for some class $d\in\cl G\setminus\{c\}$. 
\end{corollary}
\begin{proof}
  To prove the first statement, note that the map $Q\mapsto \primeins Qc$ yields a homeomorphism between
  $\spec {\Z[\xi]}$ and the irreducible component $C_c$. The second statement follows from
  Lemma~\ref{prpggr}(ii). 
  \end{proof}

The picture of crossing
irreducible components is sketched on Figure~\ref{fig:sing}.

\begin{figure}
\vspace{0.3cm}
	\begin{center}
	\begin{tikzpicture}[scale=1]
		\draw[-] (0,0) --  (4,0);
		\draw[-] (0,1) --  (4,1);
		\draw[-] (0,1.5) --  (4,2.5);
		\draw[-] (0,2.5) --  (4,1.5);
		\draw[dotted,->] (2,2) -- (2,1);
		\draw[dotted,->] (2,1) -- (2,0);
		\draw[-] (0,3) -- (2,6);
		\draw[-] (2,3) -- (4,6);
		\draw[dotted,->] (1,4.5) -- (2,2);
		\draw[dotted,->] (3,4.5) -- (2,2);
		\node at (5.5,4.5) {$\spc(\Z[\xi]^{\text{Cl}(G)})$ };
		\node at (5.5,2) { $C$ };
		\node at (5.5,1) { $\spc(\Z[\xi])$};
		\node at (5.5,0) { $\spc(\Z)$};
		\node at (2,2) {\raisebox{0ex}{\tiny$\bullet$}};
		\node at (2.17,1.75) {$P$};
		\node at (2,1) {\raisebox{0ex}{\tiny$\bullet$}};
		\node at (2.2,0.79) {$Q$};
		\node at (2,0) {\raisebox{0ex}{\tiny$\bullet$}};
		\node at (2,-0.3) {$p$};
		\node at (1,4.5) {\raisebox{0ex}{\tiny$\bullet$}};
		\node at (0.7,4.7) {$Q_{c}$};
		\node at (3,4.5) {\raisebox{0ex}{\tiny$\bullet$}};
		\node at (3.35,4.3) {$Q_{d}$};
	\end{tikzpicture}
	\end{center}
        \caption{The singularities of $\spec{\rbar G}$}
        \label{fig:sing}
        \end{figure}
Next we  recall the (standard) notions of embedding dimension and regularity at a point in a given algebraic scheme.

\begin{definition}
\label{defted:ed}
Let $X$ be a scheme over $Y$ and $P\in X$ be a point. Let also $\oo_P$ be the local ring of $P$ on $X$ and $\mmp$ its maximal ideal. Set $k_P:=\oo_P/\mmp$ to be the residue field. The (local) \emph{embedding dimension} of $X$ at $P$
is defined as
$$
{\rm edim}_P(X):=\dim_{k_P}(\mmp/\mmp^2).
$$
The point $P$ is said to be \emph{regular} (or \emph{smooth}, or \emph{nonsingular}), if $\oo_P$ is regular, i.e., 
$$
{\rm edim}_P(X)=\dim(\oo_P)
$$
where one refers to the Krull dimension on the right-hand side of the
last equation. The point is said to be \emph{singular} otherwise.
\end{definition}

From the very definition of a singular point and from Corollary~\ref{cor:comps},
we are already in position to derive some first results for the spectra
we are concerned with here.

\begin{prop} 
Let $P=\primeins Qc\in C$ be a point, with $Q\in\spec {\Z[\xi]}$ and $c\in \cl G$. Set also $p:={\rm char}(\Z[\xi]/Q)$, i.e., $(p)$ is the image of $P$ under the morphism $C\to\spec \Z$.
    \begin{itemize} 
\item[(I)] The following are equivalent:
\begin{itemize}
\item[(i)] $P$ is singular;
\item[(ii)] $P$ belongs to two distinct irreducible components
      of $C$;
    \item[(iii)] $Q\neq 0$ and there is a class $d\in\cl G\setminus\{c\}$
      such that $c^{(p)}_r=d^{(p)}_r$.
\end{itemize}
\item[(II)] if $p\nmid |G|$, then  $P$ is non-singular.
\end{itemize}
\end{prop}

\begin{proof}
Let us first show that (i) implies (ii). Any component of $C$ is isomorphic to $\spec {\Z[\xi]}$ owing to
    Corollary~\ref{cor:comps}. On the other hand, $\spec {\Z[\xi]}$ is smooth, by Lemma~\ref{lem:cyclomult}, and so if $P$ belongs to a unique component of $C$, then it is regular.
  
    Conversely, if $P$ lies in two distinct components of $C$, then
  $P$ contains two distinct minimal primes of $\rbar G$. Therefore
  the unique maximal ideal of the local ring  $\op=\rbar G_{P}$
    also contains two distinct minimal primes, which implies that
    $\op$ is not a domain. Thus $\op$ is not
    a regular local ring, which implies that $P$ is singular.

    The equivalence between (ii) and (iii) follows directly from
    Corollary~\ref{cor:comps}, while (II) is immediate from the equivalence between (i) and (iii).
\end{proof}


In the next result, we study the  (intrinsic) notion of tangency bringing  a different perspective into the study of singularities. Recall that the Zariski tangent space was defined in Definition~\ref{defted}.
This notion of tangent space should  be compared against the concept of regularity introduced  in Definition~\ref{defted:ed}. For instance, in the case of schemes over fields, the embedding dimension agrees with the dimension of the tangent, whichever the point is. On the other hand, these two invariants may differ over $\Z$. At any rate, one can easily derive the following first relations below from basic algebraic geometry.

\begin{prop}
\label{thmgr1}
Let $\pi:C\rightarrow\spc(\Z)$ be the natural projection and let $P\in C$ be a point with $\pi(P)=(p)$. Then the following hold:
\begin{itemize}
\item[(i)] $\dim_{k_P}(T_P(C/\Z))\leq \edim_P(C)$;
\item[(ii)] if $p\not\in P^2$, then $\dim_{k_P}(T_P(C/\Z))<\edim_P(C)$;
\item[(iii)] if there exists an $x\in R(G)$ with $x\not\equiv 0\,({\rm mod}\, p)$ and $x^r\equiv 0\,({\rm mod}\, p)$ with some $r\geq 2$, then  $\dim_{k_P}(T_P(C/\Z))\geq 1$ for some $P\in \pi^{-1}(p)$.
\item[(vi)] $\edim_P(C)\leq {\rm mg}(P)$, where ${\rm mg}(P)$  is
  the minimal number of generators of $P$;
\end{itemize}
\end{prop}

\begin{proof}
Set $R=\rbar G$.  To prove (i), note that $\mathcal{O}_P=R_P$ is a $\Z$-algebra, so we have the exact sequence 
\begin{equation}
\label{equses}
\mathfrak{m}_P/\mathfrak{m}_P^2\stackrel{d}\longrightarrow \Omega_{\mathcal{O}_P/\Z}\otimes_{\mathcal{O}_P}k_P\longrightarrow \Omega_{k_P/\Z}\longrightarrow 0
\end{equation}
of $k_P$-vector spaces
where ``$d$" is the derivative map. Further, the natural sequence of ring homomorphisms
$$
\Z\longrightarrow\mathbb{F}_p\longrightarrow k_P
$$
yields the following exact sequence of $k_P$-vector spaces:
\begin{equation}
\label{equfes}
\Omega_{\mathbb{F}_p/\Z}\otimes_{\mathbb{F}_p}k_P\longrightarrow \Omega_{k_P/\Z}\longrightarrow \Omega_{k_P/\mathbb{F}_p}.
\end{equation}
As $P$ is a maximal ideal of $R$, we have that $k_P\cong R/P$, thus $k_P$ is a finitely generated algebra over $\mathbb{F}_p$. Besides,  $k_P$ is a finite field
(equation~\ref{eq:kp}) and is separable over $\F_p$; therefore $\Omega_{k_P/\mathbb{F}_p}=0$. From the epimorphism $\Z\to\mathbb{F}_p$ we also get that $\Omega_{\mathbb{F}_p/\Z}=0$, hence, from (\ref{equfes}), we deduce that $\Omega_{k_P/\Z}=0$, and, from (\ref{equses}), we obtain the surjective linear map
\begin{equation}
\label{equsrj}
\mathfrak{m}_P/\mathfrak{m}_P^2\twoheadrightarrow \Omega_{\mathcal{O}_P/\Z}\otimes_{\mathcal{O}_P}k_P.
\end{equation}
Now $\Omega_{C/\spc(\Z)}\otimes_{C}k_P:=(\Omega_{C/\spc(\Z)}\otimes_{\oo_C}k_P)_P$ which is naturally isomorphic to
\[
(\Omega_{R/\Z})_P\otimes_{R_P}k_P\cong\Omega_{R_P/\Z}\otimes_{R_P}k_P=\Omega_{\oo_P/\Z}\otimes_{\oo_P}k_P.\] 
Thus, by Definition \ref{defted},
\begin{equation}
\label{equdid}
\dim_{k_P}(T_P(C/\Z))=\dim_{k_P}(\Omega_{\oo_P/\Z}\otimes_{\oo_P}k_P)
\end{equation}
and (i) follows from (\ref{equsrj}).

To prove (ii), note that if $p\not\in P^2$ then its class is nonzero in $\mmp/\mmp^2$. But the derivative map $d$ defined  in (\ref{equses}) vanishes for all elements of the image of $P\cap\Z$ in $\mmp/\mmp^2$.  In particular $d(p)=0$, but since $p$ is nonzero, we have that $\ker(d)\neq 0$. Since $d$ is surjective by (\ref{equsrj}), it follows that $\dim_{k_P}(T_P(C/\Z))<\edim_P(C)$ by (\ref{equdid}).

To prove (iii), take $x\in R(G)$ with $x\not\equiv 0$ {\rm mod} $p$ such that $x^r\equiv 0$  mod $p$ for some $r\geq 2$. Then $(1\otimes x)\otimes 1$ is a nilpotent element in $A:=(\Z[\xi]\otimes R(G))\otimes\mathbb{F}_p$. Therefore the fiber $C_{p}:=\spc\, A$ of the morphism $\pi:C\to\spc(\Z)$ over $p$ cannot be a disjoint union of the form $\bigsqcup_i k_i$ where the $k_i$ are finite separable field extensions of $\mathbb{F}_p$. From \cite[Prp.~3.2]{M} it follows that $\pi$ ramifies at $p$, and from \cite[Prp.~3.5]{M} it follows that $\Omega_{C_p/\spc(\F_p)}\not = 0$. So $\dim_{k_P}(T_{P}(C/\Z))\geq 1$ for some $P\in\pi^{-1}(p)$.

To prove (iv), take $P=(f_1,\ldots,f_m)$. Given $f\in P$ write $f=a_1f_1+\ldots+a_mf_m$ with $a_i\in R$ and taking congruence mod $P^2$ we have $
\overline{f}=\overline{a_1} \cdot \overline{f_1} + \ldots + \overline{a_m} \cdot \overline{f_m}
$	
where $\overline{a_i}\in R/P=k_P$, and $\overline{f_i}\in P/P^2$, so the $\overline{f_i}$ generate $P/P^2$. Hence
$$
\edim_P(X)=\dim(\mathfrak{m}_P/\mathfrak{m}_P^2)=\dim(P/P^2)\leqslant m
$$
and the inequality follows choosing $m$ as the minimal number of generators.
\end{proof}

\section{Computing singularities}\label{sec:comp}

In this section, we present a practical method to compute the
embedding dimension and the dimension of the Zariski tangent space at a point $P$ on the affine scheme $C=\spec{\rbar G}$ for a finite group $G$.
We also prove Theorem~\ref{th:edim}.


Given a point $P=\primeins Qc\in C$ with $Q\neq(0)$, recall from
Section~\ref{sec:cyclo} that $Q=(p,f(\xi))$ where $p\in\Z$ is a prime and $f(x)\in\Z[x]$ is an
irreducuble factor of $\Phi_n(x)$ in $\F_p[x]$. 
Let $(Q)$ denote the ideal
in $\rbar G$ generated by $Q$. As $k_P\cong \Z[\xi]/Q$ (see equation~\eqref{eq:kp}), we may consider $\rbar G/(Q)$ as a $k_P$-algebra
under the well-defined action $(r+Q)(x+(Q))=rx+(Q)$ for all $r\in\Z[\xi]$
and $x\in \rbar G$. Note that
$$
\rbar G/(Q)\cong \rbar G\otimes_{\Z[\xi]}(\Z[\xi]/Q)\cong
\rbar G\otimes_{\Z[\xi]}k_P.
$$
So set
$$
\rtilde G:= \rbar G/(Q)
$$
and note that, as $Q\subseteq P$, we have that
$$
\widetilde P:=P/(Q)=P\otimes_{\Z[\xi]}k_P
$$
is a prime ideal of $\rtilde G$.

With this in mind, we give a local description of a point of $C$ in a way that its embedding dimension can be easily computed once we study tangent spaces.

\begin{lemma}
\label{lem:primedesc}
 Let $P=\primeins Qc$ be a point of $C$ with $Q=(p,f(\xi))$. Then the following hold:
  \begin{itemize}
  \item[(i)] $P=(p,f(\xi),\chi_2-\chi_2(c),\ldots,\chi_s-\chi_s(c))$;
  \item[(ii)] $k_P\cong\rtilde G/\widetilde P $;
  \item[(iii)] $\edim_P(C)=\dim_{k_P}\widetilde P/\widetilde P^2+1$.
\end{itemize}
\end{lemma}
\begin{proof}
 For short, set $R:=\rbar G$ and let $Y:=(p,f(\xi),\chi_2-\chi_2(c),\ldots,\chi_s-\chi_s(c))$ be the ideal of $R$ we want $P$ to agree with.   Clearly $Y\subseteq P$.  As $Q\subseteq Y$, we obtain
  $$
  R/Y\cong (R/(Q))/(Y/(Q)).
  $$
  Furthermore, $(R/(Q))/(Y/(Q))$ is isomorphic to $\Z[\xi]/Q$ under the
  evaluation map
  $
  \chi+(Q)\mapsto \chi(c)+Q$.
  Therefore $R/Y$ is isomorphic to the field $\Z[\xi]/Q$.  Thus $Y$ is
  maximal and, as $Y\subseteq P$, we must have $Y=P$. This proves simultaneously~(i) and~(ii).

 Let us prove statement~(iii). Note that
  $$
  \widetilde P/\widetilde P^2=(P/(Q))/(P^2+(Q))/(Q)\cong P/(P^2+(Q)).
  $$
  Thus we obtain a chain of ideals $P\supseteq P^2+(Q)\supseteq P^2$ and
  we first prove that $\dim_{k_P}(P^2+(Q))/P^2\leq 1$. Indeed,
  by Lemma~\ref{lem:cyclomult}, $\dim_{k_P} Q/Q^2=1$. On the other hand, the map
  $q+Q^2\mapsto q+P^2$ is clearly a surjective $k_P$-homomorphism between
  $Q/Q^2\rightarrow ((Q)+P^2)/P^2$ and so $\dim_{k_P} ((Q)+P^2)/P^2\leq \dim_{k_P}
  Q/Q^2=1$.
  Now assume by seeking a contradiction that
  $\dim_{k_P}((Q)+P^2)/P^2=0$. This is equivalent to $(Q)+P^2=P^2$; that is
  $Q\subseteq P^2$. Set
  $$
  J=\{\chi\in R\mid \chi(c)\in Q^2\}.
  $$
  Then $J$ is an ideal of $R$ such that $P^2\subseteq J$. However, this
  implies that $Q\subseteq J$, and so $p=p\chi_1(c)\in Q^2$ and
  $f(\xi)=f(\xi)\chi_1(c)\in Q^2$. Thus $Q=Q^2$, which is a contradiction.
\end{proof}


Note that the ring $\rbar G$ is generated by $\chi_2,\ldots,\chi_s$ as a $\Z[\xi]$-module. So one is able to write
$$
\rbar G=\Z[\xi][x_2,\ldots,x_s]/I_G
$$
where $I_G$ is an ideal generated by some polynomials
$f_1,\ldots,f_r\in\Z[\xi][x_2,\ldots,x_s]$. The polynomials $f_i$ can be
obtained by calculating, for all $i,j$ and $k$, the coefficients
$\alpha_{i,j}^k\in\Z[\xi]$
that satisfy
\begin{equation}\label{eq:chi}
\chi_i\chi_j=\sum_{k=1}^s\alpha_{i,j}^k\chi_k.
\end{equation}
Note that $\alpha_{i,j}^k\in\Z$ for all $i,\ j,\ k$, and so
$f_i\in\Z[x_2,\ldots,x_s]$ for all $i\in\{1,\ldots,r\}$.

As above, for a fixed point $P\in\spec{\rbar G}$, consider the ring $\rtilde G$. Also recall from
Section~\ref{sec:cyclo}, that $Q=(p,f(\xi))$ and that $k_P\cong \Z[\xi]/Q$. Thus we get
$$
\rtilde G=k_P[x_2,\ldots,x_s]/(\widehat f_1,\ldots,\widehat f_r)
$$
where $\widehat f_i\in k_P[x_1,\ldots,x_s]$ is the polynomial obtained
from $f_i$ by reducing the coefficients modulo $Q$. Furthermore, as $f_i\in\Z[x_2,\ldots,x_s]$,  $\widehat f_i\in\F_p[x_2,\ldots,x_s]$.

By Lemma~\ref{lem:primedesc}, 
$$
P=(p,f(\xi),\chi_2-\chi_2(c),\ldots,\chi_s-\chi_s(c)).
$$
Setting, for $i=\{2,\ldots,s\}$, $\alpha_i=\chi_i(c)+Q$, we find that
$\alpha_i\in k_P$.
Now note that $\widetilde P=P/(Q)$ is the closed point in $\spec{\rtilde G}$ corresponding to $(\alpha_2,\ldots,\alpha_s)\in k_P^{s-1}$.
Set $\alpha=(\alpha_2,\ldots,\alpha_s)$ and 
let us define the {\em Jacobian matrix} $J_{G,P}$ and the {\em extended Jacobian matrix} $\overline J_{G,P}$ for $G$ at the point $P$ as follows:   
$$
J_{G,P}=\begin{pmatrix}
\displaystyle\frac{\partial \widehat f_1}{\partial x_2}(\alpha)
  & \cdots &
\displaystyle\frac{\partial \widehat f_1}{\partial x_s}(\alpha) \\
\vdots & \ddots & \vdots  \\
\displaystyle\frac{\partial \widehat f_r}{\partial x_2}(\alpha) &
  \cdots & \displaystyle\frac{\partial \widehat f_r}{\partial x_s}(\alpha) 
\end{pmatrix}
$$
and 
$$
\overline J_{G,P}=\begin{pmatrix}
\displaystyle\frac{\partial \widehat f_1}{\partial x_2}(\alpha)
  & \cdots &
\displaystyle\frac{\partial \widehat f_1}{\partial x_s}(\alpha) & 0 \\
\vdots & \ddots & \vdots & \vdots \\
\displaystyle\frac{\partial \widehat f_r}{\partial x_2}(\alpha) &
  \cdots & \displaystyle\frac{\partial \widehat f_r}{\partial x_s}(\alpha) & 0\\
  0 & \cdots & 0 & \widehat\Phi_n'(\xi) 
\end{pmatrix}.
$$
The matrices $J_{G,P}$ and $\overline J_{G,P}$ can be viewed as the matrices of linear transformations
$k_p^{s-1}\rightarrow k_p^{r}$ and $k_P^s\rightarrow k_P^{r+1}$, respectively, whose kernels will be  referred to as the kernel of $J_{G,P}$ and the kernel of $\overline J_{G,P}$.

The matrix $\overline J_{G,P}$ contains $J_{G,P}$ and an additional row and column in which all entries are zero except perhaps the corner entry which is $\widehat\Phi_n'(\xi)$. Noting that $\Phi_n(\xi)=0$ Lemma~\ref{lem:cyclomult} gives sufficient and necessary conditions as to when 
$\widehat \Phi_n'(\xi)=0$. Furthermore, by Theorem~\ref{th:jac}, the Jacobian matricies $J_{G,P}$ and $\overline J_{G,P}$ express the dimensions of the Zariski tangent spaces
$T_P(C/\Z[\xi])$ and $T_P(C/\Z)$, respectively. Hence the following lemma is valid.

\begin{lemma}\label{lem:tangent}
Letting $J_{G,P}$ and $\overline J_{G,P}$ be as above, the following hold.
\begin{enumerate}
    \item[(i)] $\dim_{k_P} T_P(C/\Z[\xi])=\dim_{k_P}\ker J_{G,P}$;
    \item[(ii)]  $\dim_{k_P} T_P(C/\Z)=\dim_{k_P}\ker \overline J_{G,P}$;
    \item[(iii)] $\dim_{k_P}T_P(C/\Z)=\dim_{k_P}T_P(C/\Z[\xi])$ if and only if
      either $p\nmid n$ or $p=2$ and $4\nmid n$; otherwise
    $\dim_{k_P}T_P(C/\Z)=\dim_{k_P}T_P(C/\Z[\xi])+1$.
\end{enumerate}
\end{lemma}


We are in position now to prove Theorem~\ref{th:edim}, which is a
powerful tool for computing dimensions of embedding, tangency and for relating
these concepts.

\begin{proof}[The proof of Theorem~\ref{th:edim}]
Set $\tilde P=P/(Q)=P\otimes_{\Z[\xi]} k_P$ as above. By Theorem~\ref{th:jac}, 
\begin{equation}
\label{equeqs}
\dim_{k_P}(\tilde P/\tilde P^2)=\dim_{k_P} T_P(C/\Z[\xi])=\dim_{k_P}\ker J_{G,P}.
\end{equation}
Hence we obtain that
$$
\dim_{k_P}T_P(C/\Z[\xi])\leq \dim_{k_P} T_P(C/\Z)\leq\edim_P(C)=\dim_{k_P} T_P(C/\Z[\xi)+1.
$$
  Indeed the first inequality is immediate from Lemma~\ref{lem:tangent}(iii),
  the second follows from Proposition~\ref{thmgr1}(i), and the equality follows from  Lemma~\ref{lem:primedesc}(iii) and (\ref{equeqs}). Note that one of the
  inequalities in the last displayed line must be an equality.

  If case (i) of Theorem~\ref{th:edim} holds, then 
  $\dim_{k_P}T_P(C/\Z)=\dim_{k_P}T_P(C/\Z[\xi])+1$, by Lemma~\ref{lem:tangent}(iii). Hence
        $$
      \dim_{k_P}T_P(C/\Z[\xi])+1=\dim_{k_P}T_P(C/\Z)= \edim_P(C).
      $$
      Now, if case~(ii) is valid, then $\dim_{k_P}T_P(C/\Z[\xi])= \dim_{k_P} T_P(C/\Z)$
      and so
      $$
      \dim_{k_P}T_P(C/\Z[\xi])=\dim_{k_P}T_P(C/\Z)= \edim_P(C)-1.
      $$
      
      Finally, if case~(i) holds, then $p\in Q^2$,
      by Lemma~\ref{lem:cyclomult}, and so $p\in P^2$.
      Assume that
      case~(ii) holds. Then, $Q/Q^2$ is generated by
      $p+Q^2$ (Lemma~\ref{lem:cyclomult}). Hence, if $p\in P^2$, then $(Q)+P^2=P^2$, which is impossible,
      by the argument given in the proof of Lemma~\ref{lem:primedesc}(iii).
\end{proof}

\begin{corollary}
\label{corabl}
  Suppose that $G$ is a finite  abelian group, let $C=\spec{\rbar G}$ and let
  $P\in C$ such that $p\in P\cap \Z$ for some prime $p$. Then
  $$
  \dim_{k_P} T_P(C/\Z[\xi])=\log_p(G/G^p).
  $$
\end{corollary}
\begin{proof}
  Note that the characters $\chi:G\rightarrow \C$ form a group
  (called the dual group) which is isomorphic
  to $G$. Furthermore, $G$ is the direct product
  $G=C_1\times\cdots\times C_m$
  of cyclic groups $C_i=\Z_{p_i^{\alpha_i}}$.
  Hence
  $$
  \rbar G=\Z[\xi][x_1,\ldots,x_m]/(x_1^{p_1^{\alpha_1}}-1,\ldots,
  x_m^{p_m^{\alpha_m}}-1).
  $$
  Hence the Jacobian matrix $J_{G,P}$ with respect to this presentation of $\rbar G$ over $\Z[\xi]$ is a diagonal square matrix with
  $$
  (p_1^{\alpha_1}\bar x_1^{p_1^{\alpha_1}-1},\ldots,p_m^{\alpha_m}\bar x_m^{p_m^{\alpha_m}-1})
  $$
  in the diagonal. The dimension of the kernel of this matrix (which is independent of the chosen generators and relations) is the number
  of zeros in the diagonal which is equal to the number of indices $i\in\{1,\ldots,m\}$ such
  that $p_i=p$.
  \end{proof}

The results of this section give us a practical method to calculate the dimensions of the Zariski tangent spaces $T_P(C/\Z)$ and $T_P(C/\Z[\xi])$ and the embedding dimension for $P\in C$ where 
$C=\spec{\rbar G}$ for a finite group $G$. The method can be summarized in the following simple algorithm.

\begin{algorithmic}
\Require $Q=(p,f(\xi))\in\spec{\Z[\xi]}$ and $P=\primeins Qc\in\spec{\rbar G}$
\State $\F\gets \F_p[x]/(f(x))$
\State $J\gets J_{G,P}$
\State $\dim_{k_P} T_P(C/\Z[\xi])\gets \dim\ker J$
\If{$p\nmid|G|\mbox{ or }(p=2\mbox{ and }4\nmid |G|$)}
\State $\dim_{k_P} T_P(C/\Z)\gets \dim\ker J$
\State $\edim_P\,C\gets\dim\ker J+1$
\Else \State $\dim_{k_P} T_P(C/\Z)\gets \dim_{k_P}\ker  J+1$
\State $\edim_P\,C\gets\dim\ker J+1$
\EndIf
\State \Return $\dim_{k_P} T_P(C/\Z[\xi)]$, $\dim_{k_P} T_P(C/\Z)$, $\edim_P\, C$
\end{algorithmic}

The algorithm above was implemented in the computational algebra system Magma~\cite{Magma}. The individual steps of this algorithm can be executed using built-in Magma functions. For instance the calculation of the matrix $J$ requires calculating the relations of the ring $\rbar G$ which can be done by computing the coefficients
$\alpha_{i,j}^k$ appearing in~\eqref{eq:chi}. 

\subsection{An example: $\alt 4$} To illustrate the algorithm above, consider the simple example when $G=\alt 4$. The character table of $G$ is
  $$
  \begin{array}{ccccc}
   \mbox{class} & c_1 & c_2 & c_3 & c_4\\
    \mbox{order} & 1 & 2 & 3 & 3\\
    \chi_1 & 1 & 1 & 1 & 1\\
    \chi_2 & 1 & 1& z & -1-z \\
    \chi_3 &1& 1 & -1-z & z \\
    \chi_4 &3& -1 & 0 & 0 \\
    \end{array}
  $$
  where $z$ is a primitive third root of unity. The corresponding cyclotomic polynomial
  is $\Phi_{12}(x)=x^4-x^2+1$. Also
  $$
  \Phi_{12}(x)=\left\{\begin{array}{ll}
  (x^2+x+1)^2& \mbox{in $\F_2[x]$};\\
  (x^2+1)^2& \mbox{in $\F_3[x].$}
  \end{array}\right.
  $$
  Hence $\Z[\xi]$ has two primes that project onto a prime divisor of $|G|$,
  namely,
  $$
  Q_2=(2,x^2+x+1)\qquad\mbox{and}\qquad Q_3=(3,x^2+1).
  $$
  The class $c_2$ is the unique conjugacy class which is not 2-regular while,
  $c_3$ and $c_4$ are the conjugacy classes that are not 3-regular.
  Furthermore, $(c_3)^{(3)}_r=(c_4)^{(3)}_r=c_1$.
  Hence there are two singular primes $P_2$ and $P_3$ of $\rbar G$ and they
  are given by
  \begin{eqnarray*}
    P_2&=&(2,\xi^2+\xi+1,\chi_2-1,\chi_3-1,\chi_4-1)\\
    P_3&=&(3,\xi^2+1,\chi_2-1,\chi_3-1,\chi_4).\\
    \end{eqnarray*}
  Furthermore, considering $\rbar G$ as the $\Z[\xi]$-algebra
  $\Z[\xi][x_2,x_3,x_4]/(f_1,\ldots,f_6)$, the polynomials $f_i$ are given
  by the equations
  $$
  \chi_2^2=\chi_3,
   \chi_2\chi_3=\chi_1,
    \chi_2\chi_4=\chi_4,
    \chi_3^2=\chi_2,
    \chi_3\chi_4=\chi_4,
    \chi_4^2=\chi_1+\chi_2+\chi_3+2\chi_4.
  $$
  In particular $\rbar G$ is isomorphic to $\Z[\xi][x_2,x_3,x_4]/I_G$
  where $I_G$ is the ideal generated by the polynomials
  $$
  x_2^2-x_3,\ x_2x_3-1,\ x_2x_4-x_4,\ x_3^2-x_2,\ x_3x_4-x_4,\ x_4^2-1-x_2-x_3-2x_4.
  $$
  Hence the Jacobian matrix $J=(\partial f_i)/(\partial x_j)$ with respect to these generators and relations is
  $$
  J=\begin{pmatrix}
  2x_2 & -1 & 0\\
  x_3 & x_2 & 0\\
  x_4 & 0 & x_2-1\\
  -1 & 2x_3 & 0\\
  0 & x_4 & x_3-1\\
  -1 & -1 & 2x_4-2
  \end{pmatrix}.
  $$
  Therefore we obtain the Jacobian matrices  $P_2$ and $P_3$
  over $\F_2$ and $\F_3$, respectively, by making
  the substitutions $x_2=x_3=x_4=1$ and $x_2=x_3=1$ and $x_4=0$:
  $$
  J_{G,P_2}=\begin{pmatrix}
  2 & 1 & 0\\
  1 & 1 & 0\\
  1 & 0 & 0\\
  1 & 0 & 0 \\
  0 & 1 & 0\\
  1 & 1 & 0
  \end{pmatrix}
  \quad\mbox{and}\quad
  J_{G,P_3}=\begin{pmatrix}
  2 & 2 & 0\\
  1 & 1 & 0\\
  0 & 0 & 0\\
  2 & 2 & 0\\
  0 & 0 & 0\\
  2 & 2 & 1
  \end{pmatrix}.
  $$
  Simple calculation shows that
  $$
  \dim_{\F_2}\ker J_{G,P_2}=\dim_{\F_3}\ker J_{G,P_3}=1.
  $$
  Furthermore, letting $C=\spec{\rbar G}$, we find that
  $$
  \dim_{\F_2} T_{P_2}(C/\Z[\xi])=  \dim_{\F_3} T_{P_3}(C/\Z[\xi])=1,
  $$
  $$
  \dim_{\F_2} T_{P_2}(C/\Z)=  \dim_{\F_3} T_{P_3}(C/\Z)=2,
  $$
    $$
  \edim_{P_2}C=\edim_{P_3}C=2.
  $$

We just note that this also stands for an example where the embedding dimension can be strictly smaller than the minimal number of generators. Indeed, one can check that we may shrink the number of generators here to
$$
P_2=(1+\xi+\xi^2, \chi_4-1)\quad\mbox{and}\quad P_3=(1+\xi^2, \chi_2-1,\chi_4),
$$
but not further. Thus $\edim_{P_3}(C)=2 < 3={\rm mg}(P_3)$.

  In the tables in the Appendix below,  we summarise the computations for the groups $S_3$, $D_4$, $A_4$, $D_8$, $S_4$, $A_5$ and $A_6$. The rows of the tables contain the singular primes of $\rbar G$ for each $G$. In these groups, for each prime
  $Q=(p,f(\xi))\in\spec{\Z[\xi]}$ there is a unique prime $P\in\spec{\rbar G}$ that
  projects onto $Q$.
  The Magma implementations of the procedures are available from github.com/schcs/GreenRings.

  \appendix
    \section*{Appendix: The tables}
  
    \begin{table}[!ht]\centering\begin{tabular}{|c|c|c|c|c|}\hline $p$ & $f(x)$ & 
$\edim_P\, C$ & $\dim_{k_P}T_P(C/\Z)$ & $\dim_{k_P}T_P(C/\Z[\xi])$ \\  
\hline\hline
2
&
$
x^2 + x + 1
$
&
2
&
1
&
1
\\\hline
3
&
$
x + 1
$
&
2
&
2
&
1
\\\hline
\end{tabular}\caption{The singularities of 
$\rbar{S_3}$}\label{tab:S_3}\end{table}
\begin{table}[ht]\centering\begin{tabular}{|c|c|c|c|c|}\hline $p$ & $f(x)$ & 
$\edim_P\, C$ & $\dim_{k_P}T_P(C/\Z)$ & $\dim_{k_P}T_P(C/\Z[\xi])$ \\  
\hline\hline
2
&
$
x + 1
$
&
4
&
4
&
3
\\\hline
\end{tabular}\caption{The singularities of 
$\rbar{D_4}$}\label{tab:D_4}\end{table}

\begin{table}[!ht]\centering\begin{tabular}{|c|c|c|c|c|}\hline $p$ & $f(x)$ & 
$\edim_P\, C$ & $\dim_{k_P}T_P(C/\Z)$ & $\dim_{k_P}T_P(C/\Z[\xi])$ \\  
\hline\hline
2
&
$
x^2 + x + 1
$
&
2
&
2
&
1
\\\hline
3
&
$
x^2 + 1
$
&
2
&
2
&
1
\\\hline
\end{tabular}\caption{The singularities of 
$\rbar{A_4}$}\label{tab:A_4}\end{table}

\begin{table}[!ht]\centering\begin{tabular}{|c|c|c|c|c|}\hline $p$ & $f(x)$ & 
$\edim_P\, C$ & $\dim_{k_P}T_P(C/\Z)$ & $\dim_{k_P}T_P(C/\Z[\xi])$ \\  
\hline\hline
2
&
$
x + 1
$
&
4
&
4
&
3
\\\hline
\end{tabular}\caption{The singularities of 
$\rbar{D_8}$}\label{tab:D_8}\end{table}
\begin{table}[!ht]\centering\begin{tabular}{|c|c|c|c|c|}\hline $p$ & $f(x)$ & 
$\edim_P\, C$ & $\dim_{k_P}T_P(C/\Z)$ & $\dim_{k_P}T_P(C/\Z[\xi])$ \\  
\hline\hline
2
&
$
x^2 + x + 1
$
&
3
&
3
&
2
\\\hline
3
&
$
x^2 + x + 2
$
&
2
&
2
&
1
\\\hline
3
&
$
x^2 + 2x + 2
$
&
2
&
2
&
1
\\\hline
\end{tabular}\caption{The singularities of 
$\rbar{S_4}$}\label{tab:S_4}\end{table}
\begin{table}[!ht]\centering\begin{tabular}{|c|c|c|c|c|}\hline $p$ & $f(x)$ & 
$\edim_P\, C$ & $\dim_{k_P}T_P(C/\Z)$ & $\dim_{k_P}T_P(C/\Z[\xi])$ \\  
\hline\hline
2
&
$
x^4 + x + 1
$
&
2
&
2
&
1
\\\hline
2
&
$
x^4 + x^3 + 1
$
&
2
&
2
&
1
\\\hline
3
&
$
x^4 + x^3 + 2x + 1
$
&
2
&
2
&
1
\\\hline
3
&
$
x^4 + 2x^3 + x + 1
$
&
2
&
2
&
1
\\\hline
5
&
$
x^2 + 2x + 4
$
&
2
&
2
&
1
\\\hline
5
&
$
x^2 + 3x + 4
$
&
2
&
2
&
1
\\\hline
\end{tabular}\caption{The singularities of 
$\rbar{A_5}$}\label{tab:A_5}\end{table}
\begin{table}[!ht]\centering\begin{tabular}{|c|c|c|c|c|}\hline $p$ & $f(x)$ & 
$\edim_P\, C$ & $\dim_{k_P}T_P(C/\Z)$ & $\dim_{k_P}T_P(C/\Z[\xi])$ \\  
\hline\hline
2
&
$
x^{12} + x^3 + 1
$
&
2
&
2
&
1
\\\hline
2
&
$
x^{12} + x^9 + 1
$
&
2
&
2
&
1
\\\hline
3
&
$
x^4 + x^2 + x + 1
$
&
2
&
2
&
1
\\\hline
3
&
$
x^4 + x^2 + 2x + 1
$
&
2
&
2
&
1
\\\hline
3
&
$
x^4 + x^3 + x^2 + 1
$
&
2
&
2
&
1
\\\hline
3
&
$
x^4 + 2x^3 + x^2 + 1
$
&
2
&
2
&
1
\\\hline
5
&
$
x^6 + x^3 + 2
$
&
2
&
2
&
1
\\\hline
5
&
$
x^6 + 2x^3 + 3
$
&
2
&
2
&
1
\\\hline
5
&
$
x^6 + 3x^3 + 3
$
&
2
&
2
&
1
\\\hline
5
&
$
x^6 + 4x^3 + 2
$
&
2
&
2
&
1
\\\hline
\end{tabular}\caption{The singularities of 
$\rbar{A_6}$}\label{tab:A_6}\end{table}


\end{document}